\documentclass{svproc}
\usepackage[numbers,sort&compress]{natbib}
\usepackage{cite}
\usepackage{array}
\usepackage{url}
\usepackage{empheq}
\usepackage{float}
\usepackage{comment}
\usepackage{amsfonts,amssymb,amsbsy,amsmath,amsthm}
\usepackage{booktabs}
\usepackage{xcolor}
\usepackage{float}
\usepackage{multirow}
\usepackage{hyperref}
\usepackage{indentfirst} 

\begin{document}
\mainmatter

\title{Multicriteria Portfolio Selection with Intuitionistic Fuzzy Goals as a Pseudoconvex Vector Optimization}
\titlerunning{Multicriteria Portfolio Selection with Intuitionistic Fuzzy Goals}  
\author{Vuong D. Nguyen\inst{1}, Nguyen Kim Duyen\inst{2}, Nguyen Minh Hai \inst{3} \\\and Bui Khuong Duy*\inst{4}}
\authorrunning{Vuong D. Nguyen et al.} 
%
\institute{Department of Computer Science, University of Houston, Houston TX, USA\\
\email{dnguyen170@uh.edu}
\and 
Faculty of Banking and Finance, Foreign Trade University, Hanoi, Vietnam
\email{822407@ftu.edu.vn}
\and
Hanoi University of Science, Vietnam National University, Hanoi, Vietnam
\email{nguyenminhhai\textunderscore hsgs2022@hus.edu.vn}
\and
School of Applied Mathematics and Informatics, Hanoi University of Science and Technology, Hanoi, Vietnam.\\
*Corresponding author, Email: \email{duy.bk195864@sis.hust.edu.vn}
}

\maketitle
\begin{abstract}
Portfolio selection involves optimizing simultaneously financial goals such as risk, return and Sharpe ratio. This problem holds considerable importance in economics. However, little has been studied related to the nonconvexity of the objectives. This paper proposes a novel generalized approach to solve the challenging Portfolio Selection problem in an intuitionistic fuzzy environment where the objectives are soft pseudoconvex functions, and the constraint set is convex. Specifically, we utilize intuitionistic fuzzy theory and flexible optimization to transform the fuzzy pseudoconvex multicriteria vector into a pseudoconvex programming problem that can be solved by recent gradient descent methods. We demonstrate that our method can be applied broadly without special forms on membership and nonmembership functions as in previous works. Computational experiments on real-world scenarios are reported to show the effectiveness of our method.
\end{abstract}

\section{Introduction}
Portfolio selection is crucial for managing investment risks and optimizing returns. It involves allocating investment assets to achieve specific investment goals, such as maximizing returns or minimizing risks. This requires decisions on the allocation of the weight of assets for different investments, including stocks, bonds, cash, and other assets. Portfolio optimization tools have received attention and development to automate the asset allocation process.

Markowitz's Portfolio Optimization model, alternatively named the Mean-Variance model, is a highly prevalent portfolio optimization model. Investors want to maximize the expected total return while keeping the portfolio's volatility to a minimum or at a certain threshold, leading to a bi-criteria convex optimization problem. Later researchers widely extended Markowitz's model by taking into account risk aversion index \citep{Thang2021}, value-at-risk \citep{Khanjani2020} or Skewness \citep{Jagdish2021}. In this paper, we further consider the Sharpe ratio because it has the function with necessary properties that suits the main problem proposed as shown in \citep{Vuong2023}. Sharpe ratio is an important factor that measures the performance of the portfolio via the ratio of expected return to standard deviation.

 Traditional portfolio selection models based on crisp (i.e., deterministic) optimization techniques have several limitations, such as the inability to capture uncertainties in the decision-making process of investors. Thus, in this study, we propose to solve a more practical portfolio selection problem in an intuitionistic fuzzy environment with soft goals (IFPS). This intuitionistic fuzzy multicriteria programming problem (IFMOP) has been solved in simple scenarios of linear objectives and convex objectives (see \citep{yu2021intuitionistic, gupta2020intuitionistic, Sakawa2013}). Later works 
approached Problem (IFMOP) by interactive programming that requires frequent interaction with the decision makers and hard constraints on the activation of the fuzzy objective.
 Furthermore, little work has been proposed to solve Problem (IFMOP).
 
In this research, we extend Markowitz's model by considering Sharpe Ratio. Specifically, we first propose the intuitionistic fuzzy multicriteria portfolio selection problem. We then present the construction of our method to solve the equivalent (IFMOP) in a generalized framework based on the nice property of the Sharpe ratio function. We utilize flexible optimization to transform the original fuzzy multicriteria optimization problem to a pseudoconvex programming problem, then we exploit the algorithm in \citep{Thang2022self} to solve the equivalent problem. Unlike previous works, our method not only requires the least interaction with the decision makers but also works effectively without any hard assumption on activating the soft goals. This
makes our method more flexible and robust, enabling to handle imprecise and uncertain data, multiple conflicting objectives, and various constraints.

We organize the remaining of the article in 4 sections as follows. In Section
2, we introduce the main portfolio selection problem. Section 3 presents the preliminaries and our methodology to solve the equivalent problem (IFMOP). Section 4 demonstrates the effectiveness of our method via experiments on real-world portfolio selection problems. Section 5 presents our conclusion.

\section{Multicriteria Portfolio Selection Problem}
Consider a portfolio vector $x = (x_1,\dots, x_n)$ where $x_k$ is the proportion invested in $k^{th}$ asset. In reality, there is a constraint set on $x$, and we denote it as $\mathcal{X} = \{ x \in \mathbb{R}_{+}^{n} \mid x_1 + \dots + x_n = 1\}$.  

Let there be $n$ assets with random returns represented by a random vector $\mathcal{R} = (\mathcal{R}_1, \mathcal{R}_2, \dots, \mathcal{R}_n)^T$, and the expected returns of those $n$ assets are denoted by vector $\mathcal{L} = (\mathcal{L}_1, \mathcal{L}_2, \dots, \mathcal{L}_n)^T$ . Thus, the total random return of $n$ assets is represented by $\mathcal{R}^Tx = \sum_{k=1}^n \mathcal{R}_k x_k $, which is a linear stochastic function. However, investors typically consider the expected returns of $n$ asset classes as follows
\begin{equation}
\mathcal{E}(x)=E(\mathcal{R}^{T}x)=\sum_{k=1}^{n}\mathcal{L}_{k}x_{k}.    
\end{equation}

Let $\mathcal{Q}=(\sigma_{ij})_{n\times n}$ be the covariance matrix of random vector $\mathcal{R}$.
Then, the variance of returns, i.e. risk of the portfolio, can be denoted as
\begin{equation}
    \mathcal{V}(x)=Var(\mathcal{R}^{T}x)=\sum_{i=1}^{n}\sum_{k=1}^{n}\sigma_{ik}x_{i}x_{k}
\end{equation}
where, $\sigma_{ii}^{2}$ represents the variance of $\mathcal{R}_i$, and $\sigma_{ik}$ denotes the correlation
coefficient between $\mathcal{R}_{k}$ and $\mathcal{R}_{i}$, $i,k=1,2,...,n$.

Beyond return and risk, investors demand to understand the return of an investment compared to its risk, which is represented as the Sharpe ratio
\begin{equation}
    \mathcal{S}r(x)=\frac{\mathcal{E}(x)-p_{rf}}{\sqrt{\mathcal{V}(x)}}    
\end{equation}
where $\mathcal{S}r(x)$ denotes the Sharpe ratio, $p_{rf}$ denotes the rate of a zero-risk-portfolio's return. The benchmark in return is then divided by $\sqrt{\mathcal{V}(x)}$ which measures how much the portfolio excesses standard deviation of return.

According to Markowitz's model, the investor wants to optimize two goals
\begin{align}
\begin{array}{ll}\tag{MV}\label{bt_MV}
   \operatorname{Max}  & \mathcal{E}(x)\\
  \operatorname{Min}  &  \mathcal{V}(x)\\
\text{s.t. } & x \in \mathcal{X}.  
\end{array}
\end{align}
where objective $\mathcal{E}(x)$ is a linear function and objective $\mathcal{V}(x)$ is a convex function.

In this research, by rewriting $\mathcal{E}^*(x) = -\mathcal{E}(x)$ and $\mathcal{S}r^*(x) = -\mathcal{S}r(x)$, we propose the tri-criteria vector minimization problem as follows
\begin{align}
\begin{array}{ll}\tag{MVS}\label{bt_MVS}
    \operatorname{Min} & \{\mathcal{E}^*(x), \mathcal{V}(x), \mathcal{S}r^*(x) \} \\
    \text{s.t. } & x \in \mathcal{X}.
\end{array}
\end{align}
where $\mathcal{X}$ is a non-empty convex set. The property of $\mathcal{S}r^*(x)$ which plays a key role in the construction of our method will be presented in Section 3.

\section{Multicriteria Portfolio Selection with Intuitionistic Fuzzy Goals}
\subsection{Intuitionistic fuzzy goals}

To amplify the uncertainty in Fuzzy portfolio selection problem, intuitionistic fuzzy goals allow decision makers to express ambiguity in their goals. Given an universal set $\mathbb X$, the generalization of fuzzy sets that allow for more nuanced and flexible representation of uncertainty is called an intuitionistic fuzzy set of $\mathbb X$, denoted by $\Tilde{\mathcal{A}}$. Consider the following two mappings:
the membership mapping $\mu_{\Tilde{\mathcal{A}}}: \mathbb X \rightarrow [0, 1]$
and the non-membership mapping $ \nu_{\Tilde{\mathcal{A}}}: \mathbb X\rightarrow [0,1]$, we present the definition of the intuitionistic fuzzy set $\Tilde{\mathcal{A}}$ as follows  

\begin{definition}\citep{atanassov1983intuitionistic}\label{def:IFS} Conditioning $0 \leq \mu_{\Tilde{\mathcal{A}}}(x) + \nu_{\Tilde{\mathcal{A}}}(x) \leq 1$ for all $x \in\mathbb X$, then

\begin{equation} \Tilde{\mathcal{A}} = \{ \left<x, \mu_{\Tilde{\mathcal{A}}}(x), \nu_{\Tilde{\mathcal{A}}}(x) \right> \mid x\in\mathbb X \},\end{equation}

\end{definition}
The numbers $\mu_{\Tilde{\mathcal{A}}}(x)$ and $\nu_{\Tilde{\mathcal{A}}}(x)$ represent the membership and non-membership of $x$ in $\Tilde{\mathcal{A}}$, respectively.

Now, the general multi-objective vector optimization problem with the following formula is now being considered
\begin{align}
\begin{array}{ll}\tag{MOP}\label{bt_GMOP}
    \operatorname{Min} & \mathcal{F}(x)=(\mathcal{F}_{1}(x),...,\mathcal{F}_{k}(x))^T \\
    \text{s.t.} & x\in \mathcal{X}\nonumber
\end{array}
\end{align}
where $\mathcal{X}$ is a nonempty compact convex set. General fuzzy optimization refers to the formulation of optimization problems using fuzzy sets, where the constraints and objectives are flexible, approximate, or uncertain. 
We consider the main problem with fuzzy form as
\begin{align}
\begin{array}{ll}\tag{IFMOP}\label{bt_FGMOP}
  \operatorname{\widetilde{Min}} & \mathcal{F}(x)=(\mathcal{F}_{1}(x),...,\mathcal{F}_{k}(x))^T \\
\text{s.t. } & x\in \mathcal{X}\nonumber 
\end{array} 
\end{align}
where $\widetilde{\operatorname{Min}}$ represents ``to minimize as well as possible based on the demand of the decision makers''. The approach to problem (\ref{bt_GMOP}) with fuzzy objectives
is widely applied by decision-makers in many real-world  problems.

In this study, we propose models based on an intuitionistic fuzzy set.  Therefore, we use the membership and non-membership functions to associate the input data. These functions are vital in intuitionistic fuzzy optimization. Consider a monotonic decreasing function $m_i(\cdot)$,
the membership function $\mu_i$ with respect to $\mathcal{F}_i$ has the form as following
\begin{align}
\mu_{i}\left(\mathcal{F}_{i}(x)\right)=\begin{cases}
0 & \text{if }\mathcal{F}_{i}(x)\geq y_{i}^{0},\\
m_{i}(\mathcal{F}_i(x)) & \text{if }y_{i}^{0}\geq \mathcal{F}_{i}(x)\geq y_{i}^{1},\\
1 & \text{if }\mathcal{F}_{i}(x)\leq y_{i}^{1},
\end{cases}    
\end{align}
where $y_{i}^{0}$ is the minimum value of $\mathcal{F}_{i}$
if $\mu_{i}\left(\mathcal{F}_{i}(x)\right)=0$ and $y_{i}^{1}$ is the maximum value of $f_{i}$
if $\mu_{i}\left(\mathcal{F}_{i}(x)\right)=1$. 
On the other hand, consider a monotonic increasing $n_i(\cdot)$, the non-membership function $\nu_i$ with respect to $\mathcal{F}_i$ has the form 
\begin{align}
    \nu_{i}\left(\mathcal{F}_{i}(x)\right)=\begin{cases}
1 & \text{if }\mathcal{F}_{i}(x)\geq y_{i}^{0},\\
n_{i}(\mathcal{F}_i(x)) & \text{if }y_{i}^{0}\geq \mathcal{F}_{i}(x)\geq y_{i}^{1},\\
0 & \text{if }\mathcal{F}_{i}(x)\leq y_{i}^{1},
\end{cases} 
\end{align}
where $y_{i}^{0}$ is the minimum value of $\mathcal{F}_{i}$
if $\nu_{i}\left(\mathcal{F}_{i}(x)\right)=1$ and $y_{i}^{1}$ is the maximum value of $\mathcal{F}_{i}$ if $\nu_{i}\left(\mathcal{F}_{i}(x)\right)=0$. 

The mappings $\mu_{i}\left(\mathcal{F}_{i}(x)\right)$ and $\nu_i(\mathcal{F}_i(x))$ are also called intuitionistic fuzzy mappings which map $\mathcal{F}_i(x)$ to an intuitionistic fuzzy number belonging to interval $[0,1]$ and satisfy $0\leq \mu_{i}\left(\mathcal{F}_{i}(x)\right) + \nu_i(\mathcal{F}_i(x)) \leq 1$ according to definition \ref{def:IFS}. Several conditions on characteristics of a general fuzzy mapping have been proposed in \citep{Gomez2015}. However, we propose the following proposition about the properties of intuitionistic fuzzy mappings used in the problem model in the following section.
\begin{proposition}
    \label{prop:mu}
$\mu_{i}$, $i = 1,\ldots,k$ is monotonic decreasing
function, and $\nu_i$, $i=1,\ldots, k$ is monotonic increasing function.
\end{proposition}
By utilizing Proposition \ref{prop:mu}, we propose to build intuitionistic fuzzy mappings that require the least interaction with the decision maker and also present our novel method to transform the intuitionistic fuzzy multicriteria decision problem into a deterministic problem that can be solved effectively using pseudoconvex programming algorithms. The same scheme can be applied to the picture fuzzy set (see \citep{van2022applied,van2023applied,long2022novel}) to build the multicriteria  portfolio selection with picture fuzzy goals. This expansion will be developed by us in the future. 

\subsection{Transformation to deterministic model}
Recall problem (\ref{bt_MVS}). As mentioned above, this problem is not multiobjective convex programming. However, we rely on the nice property of $\mathcal{S}r(x)$ to prove its property, as below
\begin{definition}\textbf{(Pseudoconvex function (see \citep{mangasarian1975pseudo})).} Given a non-empty convex set $\mathbb X$ and a differentiable function $f: \mathbb R^n \to \mathbb R$ on $\mathbb X$. We say that $f$ is a pseudoconvex function on $\mathbb X$ if for all $x^1, x^2$ in $\mathbb X$, it holds that:
    \begin{equation}
        f(x^2) < f(x^1) \Rightarrow \langle\nabla f(x^1), x^2-x^1\rangle < 0.
    \end{equation}  
    If $f$ is a pseudoconvex function, then $-f$ is called a pseudoconcave function.
\end{definition}
The following proposition will be verified to us about the pseudoconcavity of the function Sharpe ratio
\begin{proposition} \label{prop2}
    $\mathcal{S}r(x)$ is a pseudoconcave function, and $\mathcal{S}r^*(x)$ is a pseudoconvex function.
\end{proposition}
\begin{proof}
    Note that $\mathcal{E}(x) - p_{rf}$ is positive linear as $p_{rf}$ is a constant, while $\sqrt{\mathcal{V}(x)}$
is  convex. On the other hand, given two functions $\varphi_{1}$
and $\varphi_{2}$ defined on a set $X$, if $\varphi_{1}$ is a positive and concave functions,
$\varphi_{2}$ is a positive convex function on $X$ satisfied $\varphi_{1}, \varphi_{2}$ are differentiable functions on $X$, then fractional function $\frac{\varphi_{1}}{\varphi_{2}}$
is pseudoconcave function on $X$ (see in \citep{avriel1988generalized}). So, the function $\mathcal{S}r(x)$ is a pseudoconcave function, and $\mathcal{S}r^*(x) = -\mathcal{S}r(x)$ is a pseudoconvex function.
\end{proof}
From Proposition \ref{prop2}, we have shown that the Problem (\ref{bt_MVS}) is a multiobjective pseudoconvex programming problem. We consider this problem in an intuitionistic fuzzy environment. Then the problem has the formulation as
\begin{align}
    \begin{array}{ll}\tag{IFMVS}\label{bt_FMVS}
        \widetilde{\operatorname{Min}} & \{\mathcal{E}^*(x), \mathcal{V}(x), \mathcal{S}r^*(x) \} \\
    \text{s.t.} & x \in \mathcal{X}.
    \end{array}
\end{align}

Associating the input data requires using membership and non-membership functions. Therefore, it is needed to develop a way to define the functions $\mu_i$ and $\nu_i$ according to definition \ref{def:IFS} and proposition \ref{prop:mu}. Prior studies have assumed that decision-makers engage in frequent interaction, i.e. the decision-makers can specify $y_{i}^{1}$ and
$y_{i}^{0}$ within $y_{i}^{\text{min}}$ and $y_{i}^{\text{max}}$ \citep{Sakawa2013}. Furthermore, previously proposed methods are restricted to apply only certain types of fuzzy mappings, such as the popular linear monotonic decreasing and increasing mappings given as
\begin{equation}
\mu_{i}^{L}(\mathcal{F}_i(x))=\frac{y_i^0-\mathcal{F}_i(x)}{y_{i}^{0}-y_{i}^{1}} \hspace{0.8cm}\text{ and }\hspace{0.8cm} \nu_i^L(\mathcal{F}_i(x)) = \frac{\mathcal{F}_i(x) - y^1_i}{y^0_i - y^1_i} \label{linearmf}
\end{equation}
We instead propose a method that can be effectively applied with any arbitrary membership function. More importantly, our framework adapts to a more realistic environment where frequent interaction with the decision-makers is not feasible. Specifically, we propose to find the appropriate range for $y_{i}^{0}$ and $y_{i}^{1}$, then use the range to ease the process of acquiring requirements from the decision-makers. Consider the following problems
\begin{equation}
\operatorname{min}_{x\in \mathcal{X}}\mathcal{F}_{i}(x),\,i = 1,\ldots,k,\tag{\ensuremath{P_{i}^{m}}}\label{bt_Pm}
\end{equation}
\begin{equation}
\operatorname{max}_{x\in \mathcal{X}}\mathcal{F}_{i}(x),\,i = 1,\ldots,k.\tag{\ensuremath{P_{i}^{M}}}\label{bt_PM}
\end{equation}
Denote $y_{i}^{\text{min}}$ as the optimal solution to problem (\ref{bt_Pm}) and $y_{i}^{\text{max}}$ as an upper bound of the problem (\ref{bt_PM}). 
\begin{remark}
\label{rem:1} Consider a pseudoconvex programming
problem, if $\hat{x}$ is locally optimal, then $\hat{x}$ is globally optimal (see \citep{Mangasarian1994}).
\end{remark}
By remark (\ref{rem:1}), $y_{i}^{\text{min}}$ can be easily found. To deal with problem (\ref{bt_PM}) which is a non-convex problem,
instead, find an upper bound $y_{i}^{\text{max}}$ of this problem
(see \citep{Benson1998,thang2020monotonic} for details and illustration). Then, for $\,i = 1,\ldots,k$, $y_{i}^{0}$
and $y_{i}^{1}$ are calculated by
\begin{equation}
y_{i}^{1}=y_{i}^{\text{min}}, y_{i}^{0}=y_{i}^{\text{max}} \label{eq:z1}
\end{equation}
Now, we can rewrite the deterministic form of the problem  (\ref{bt_FMVS}) as
\begin{align}
    \begin{array}{ll}\tag{BFMVS}\label{bt_BFMVS}
        \operatorname{Max} & \{\mu_1(\mathcal{E}^*(x)), \mu_2(\mathcal{V}(x)), \mu_3(\mathcal{S}r^*(x))\} \\
        \operatorname{Min} & \{\nu_1(\mathcal{E}^*(x)), \nu_2(\mathcal{V}(x)), \nu_3(\mathcal{S}r^*(x))\} \\
    \text{s.t. } & x \in \mathcal{X}.
    \end{array}
\end{align}
where $\mu_i, \nu_i$ are the mappings of the intuitionistic fuzzy set corresponding to $\mathcal{E}^*(x), \mathcal{V}(x), \mathcal{S}r^*(x)$. 
Set $\eta_i(\cdot)=1-\mu_i(\cdot)$, then problem (\ref{bt_BFMVS}) can be rewritten as
\begin{align}
\begin{array}{ll}\tag{MFMVS}\label{bt_MFMVS}
    \operatorname{Min} & \{\eta_1(\mathcal{E}^*(x)),\eta_2(\mathcal{V}(x)), \eta_3(\mathcal{S}r^*(x)), \nu_1(\mathcal{E}^*(x)), \nu_2(\mathcal{V}(x)), \nu_3(\mathcal{S}r^*(x))\} \\
    \text{s.t. } & x \in \mathcal{X}.
\end{array}
\end{align}
Instead of directly solving this problem, we propose to convert problem (\ref{bt_MFMVS}) into the following problem
\begin{align}
\begin{array}{ll}\tag{TFMVS}\label{bt_TFMVS}
    \operatorname{Min} & \operatorname{max} \{\eta_1(\mathcal{E}^*(x)),\eta_2(\mathcal{V}(x)), \eta_3(\mathcal{S}r^*(x)), \nu_1(\mathcal{E}^*(x)), \nu_2(\mathcal{V}(x)), \nu_3(\mathcal{S}r^*(x))\} \\
    \text{s.t. } & x \in \mathcal{X}.
\end{array}
\end{align}
This single objective problem has received great attention from many authors (see in \citep{tuoi2022fuzzy, hoang2021stochastic}). However, previous works only deal with objectives without pseudoconvexity. In this paper, we demonstrate the properties of this problem as follows
\begin{proposition}\label{prop:pseudo}
(\ref{bt_TFMVS}) is a pseudoconvex programming problem.
\end{proposition}
For the pseudoconvex programming problem, Thang et al. \citep{Thang2022self} have demonstrated the convergence of the global solution and proposed a solution algorithm based on the gradient direction method. The efficiency of this algorithm has been shown by computational results. Consequently, we propose to use it to solve the problem (\ref{bt_TFMVS}).

\section{Computational Experiment}

In this section, to show the effectiveness of our proposed method in real-world scenarios, consider the following portfolio selection problem.
\begin{example} Consider 7 different stocks as a portfolio with expected return given in table \ref{table1} and covariance matrix given in table \ref{table2}.
\end{example}
\begin{table}[h]
\centering
\begin{tabular}{cccccccc}
\toprule 
stock & $\mathcal{S}_t\mathcal{C}_1$ & $\mathcal{S}_t\mathcal{C}_2$ & $\mathcal{S}_t\mathcal{C}_3$ & $\mathcal{S}_t\mathcal{C}_4$ & $\mathcal{S}_t\mathcal{C}_5$ & $\mathcal{S}_t\mathcal{C}_6$ & $\mathcal{S}_t\mathcal{C}_7$\tabularnewline
\midrule
\midrule 
Expected return & $0.0282$ & $0.0462$ & $0.0188$ & $0.0317$ & $0.01536$ & $0.0097$ & $0.01919$\tabularnewline
\bottomrule
\end{tabular}
\vspace{0.2cm}
\caption{\label{table1}Expected return of each stock}
\end{table}

\begin{table}[h]
\centering
\begin{tabular}{cccccccc}
\toprule 
stock &$\mathcal{S}_t\mathcal{C}_1$ & $\mathcal{S}_t\mathcal{C}_2$ & $\mathcal{S}_t\mathcal{C}_3$ & $\mathcal{S}_t\mathcal{C}_4$ & $\mathcal{S}_t\mathcal{C}_5$ & $\mathcal{S}_t\mathcal{C}_6$ & $\mathcal{S}_t\mathcal{C}_7$\tabularnewline
\midrule
\midrule 
$\mathcal{S}_t\mathcal{C}_1$ & $0.0119$ & $0.0079$ & $0.0017$ & $0.0019$ & $0.0022$ & $-0.0008$ & $0.0032$\tabularnewline
\midrule 
$\mathcal{S}_t\mathcal{C}_2$ & $0.0079$ & $0.0157$ & $0.0016$ & $0.0013$ & $0.0005$ & $-0.0026$ & $0.0035$\tabularnewline
\midrule 
$\mathcal{S}_t\mathcal{C}_3$ & $0.0017$ & $0.0016$ & $0.0056$ & $-0.0002$ & $0.0030$ & $0.0017$ & $-0.0003$\tabularnewline
\midrule 
$\mathcal{S}_t\mathcal{C}_4$ & $0.0019$ & $0.0013$ & $-0.0002$ & $0.0093$ & $-0.0007$ & $0.0010$ & $0.0024$\tabularnewline
\midrule 
$\mathcal{S}_t\mathcal{C}_5$ & $0.0022$ & $0.0005$ & $0.0030$ & $-0.0007$ & $0.0110$ & $0.0010$ & $0.0011$\tabularnewline
\midrule 
$\mathcal{S}_t\mathcal{C}_6$ & $-0.0008$ & $-0.0026$ & $0.0017$ & $0.0010$ & $0.0010$ & $0.0067$ & $0.0014$\tabularnewline
\midrule 
$\mathcal{S}_t\mathcal{C}_7$ & $0.0032$ & $0.0035$ & $-0.0003$ & $0.0024$ & $0.0011$ & $0.0014$ & $0.0130$\tabularnewline
\bottomrule
\end{tabular}
\vspace{0.2cm}
\caption{\label{table2}Covariance matrix of chosen stocks}
\end{table}
We solved two problems (\ref{bt_MV}) and (\ref{bt_MVS}) and their intuitionistic fuzzy versions. By denoting $\overline{x}$ as the solution of the original unfuzzy problem and $x^f$ as the solution of the intuitionistic fuzzy version, the computational results are presented as follows. Table \ref{table3} shows the value of optimal solutions $\overline{x}$ and $x^f$ to both problems (\ref{bt_MV}) and (\ref{bt_MVS}), while Table \ref{table4} demonstrates an important insight into the values of $\mathcal{E}(x), \mathcal{V}(x)$ and $\mathcal{S}r(x)$ yielded by unfuzzy and fuzzy solutions.
\begin{table}[h]
\centering
\begin{tabular}{ccc}
\toprule 
Problem & Sol & Value\tabularnewline
\midrule
\midrule 
\multirow{2}{*}{MV} &  $\overline{x}$ & $(0.0287, 0.1150,0.2274, 0.1857, 0.1111,  0.2653,
 0.0668)$\\
\cline{2-3}
& $x^f$& $(0.1078, 0.1268, 0.1740, 0.1526, 0.1257, 0.1981, 0.1150)$\tabularnewline
\midrule
\multirow{2}{*}{MVS}& $\overline{x}$ & $(0.0289, 0.1147, 0.2274, 0.1857, 0.1111, 0.2654,
 0.0668)$\\
\cline{2-3}
& $x^f$ & $(0.1026, 0.3680, 0.1016, 0.1265, 0.1006, 0.1000,
 0.1007)$\tabularnewline
\bottomrule
\end{tabular}
\vspace{0.2cm}
\caption{\label{table3}Fuzzy optimal solutions to (\ref{bt_MV}) 
and (\ref{bt_MVS})}
\end{table}
\begin{table}[h]
\centering
\begin{tabular}{ccccc}
\toprule 
Problem & Solution & $\mathcal{E}(x)$ & $\mathcal{V}(x)$ & $\mathcal{S}r(x)$\tabularnewline
\midrule
\midrule
\multirow{2}{*}{MV} & $\overline{x}$ & 0.02184&0.0022 &-\\
\cline{2-5}
& $x^f$&0.0230 &0.0024 &-\tabularnewline
\midrule
\multirow{2}{*}{MVS}& $\overline{x}$ &0.02183 &0.0022 &0.3562\\
\cline{2-5}
& $x^f$ &0.0302 &0.0041 & 0.3938\tabularnewline
\bottomrule
\end{tabular}
\vspace{0.2cm}
\caption{\label{table4}$\mathcal{E}(x), \mathcal{V}(x) \text{ and } \mathcal{S}r(x)$ values of the solution}
\end{table}

From Tables \ref{table3}, it can be seen that the solutions to unfuzzy Problem (\ref{bt_MV}) and Problem (\ref{bt_MVS}) are stable, leading to a very slight difference in the optimal values of $\mathcal{E}(x)$ and $\mathcal{V}(x)$ even when  $\mathcal{S}r(x)$ is taken into consideration as shown in Table \ref{table4}. This shows that restricting to crisp objectives severely undermines the impact of Sharpe ratio on expected return and risk. More importantly, from Table \ref{table4}, when enabling soft goals, we can observe a remarkable increase in the optimal values of expected return $\mathcal{E}(x)$ and risk $\mathcal{V}(x)$ in the intuitionistic fuzzy version. Specifically, the expected return is significantly larger, i.e. $0.0302$, compared to $0.02184$, which shows that our proposed model actually softens and effectively exploit the widened goals to achieve a better outcome than the original rigid versions. The optimal value of Sharpe ratio from our proposed Problem (\ref{bt_MVS}) is also reported, which helps investors make better investment decisions in terms of expected return per unit of risk. Moreover, it is also shown in Table \ref{table4} that investors may flexibly accept higher risk.

\section{Conclusion}
In this article, we consider a generalized multicriteria portfolio selection model and examine them in an intuitionistic fuzzy environment. We first show that this problem is a pseudoconvex programming problem. We then propose a method to convert the problem to the equivalent single-criteria problem. Our method outperforms previous approaches in both simplicity and effectiveness, where no assumption needs to be made on the frequent interaction with the investors and no hard constraint is set on the membership function. Finally, the results show that the intuitionistic fuzzy multicriteria portfolio selection problem that we propose has given us other better options regarding actual return expectations. In future work, we want to use algorithms to solve the (\ref{bt_MFMVS}) problem on an effective set of solutions \citep{tran2023framework, thang2015outcome, thang2016solving}, from which investors have more recommendations to choose from when deciding on asset allocation in the actual investment.

\renewcommand\bibname{References}
\bibliographystyle{splncs03}
\bibliography{references}
\end{document}